\documentclass[11pt]{amsart}
\usepackage{amscd,amsmath,amssymb,amsfonts}
\usepackage{color}
\usepackage[cmtip, all]{xy}
\usepackage{graphicx}
\usepackage{enumerate}
\newtheorem{thm}[equation]{Theorem}

\newtheorem{lemma}[equation]{Lemma}

\newtheorem{prop}[equation]{Proposition}

\newtheorem{definition}[equation]{Definition}
\newtheorem{conjecture}[equation]{Conjecture}
\newtheorem{cor}[subsection]{Corollary}
\newtheorem{hyp}[equation]{Hypothesis}
\newtheorem{question}[equation]{Question}
\newtheorem{defn}[subsection]{Definition}

\newtheorem{step}[subsection]{Step}
\newtheorem*{thm*}{Theorem}
\theoremstyle{remark}

\newtheorem{remark}[equation]{Remark}

\numberwithin{equation}{section}

\newcommand{\isoarrow}{{~\overset\sim\longrightarrow~}}

\newcommand{\Gm}{{\mathbb G}_m}

\newcommand{\CG}{{\mathcal{G}}}
\newcommand{\C}{{\mathcal{C}}}

\newcommand{\CA}{{\mathcal{A}}}

\newcommand{\CL}{{\mathcal{L}}}

\newcommand{\G}{{\Gamma}}

\newcommand{\ra}{{~\rightarrow~}}

\newcommand{\Qlb}{{\overline{\mathbb Q}_{\ell}}}

\newcommand{\Fp}{{{\mathbb F}_p}}

\newcommand{\ad}{{\mathbf A}}

\begin{document}
\author{Michael Harris}
\thanks{This research received funding from the European Research Council under the European Community's Seventh Framework Programme (FP7/2007-2013) / ERC Grant agreement no. 290766 (AAMOT).  The author was partially supported by NSF Grant DMS-1701651.
}

\address{Michael Harris\\
Department of Mathematics, Columbia University, New York, NY  10027, USA}
 \email{harris@math.columbia.edu}

\date{\today}

\title[Incorrigible representations]{Incorrigible representations}

\maketitle

\begin{abstract}
As a consequence of his numerical local Langlands correspondence for $GL(n)$, Henniart deduced the following theorem: If $F$ is a nonarchimedean local field and if $\pi$ is an irreducible admissible representation of $GL(n,F)$, then, after a finite sequence of cyclic base changes, the image of $\pi$ contains a vector fixed under an Iwahori subgroup. This result was indispensable in all proofs of the local Langlands correspondence. Scholze later gave a different proof, based on the analysis of nearby cycles in the cohomology of the Lubin-Tate tower.

Let $G$ be a reductive group over $F$. Assuming a theory of stable cyclic base change exists for $G$, we define an incorrigible supercuspidal representation $\pi$ of $G(F)$ to be one with the property that, after any sequence of cyclic base changes, the image of $\pi$ contains a supercuspidal member. If $F$ is of positive characteristic then we define $\pi$ to be pure if the Langlands parameter attached to $\pi$ by Genestier and Lafforgue is pure in an appropriate sense. We conjecture that no pure supercuspidal representation is incorrigible. We sketch a proof of this conjecture for $GL(n)$ and for classical groups, using properties of standard $L$-functions; and we show how this gives rise to a proof of Henniart's theorem and the local Langlands correspondence for $GL(n)$ based on V. Lafforgue's Langlands parametrization, and thus independent of point-counting on Shimura or Drinfel'd modular varieties.

This paper is an outgrowth of the author's paper arXiv:1609.03491 with G. B\"ockle, S. Khare, and J. Thorne.
\end{abstract}

\tableofcontents


\section{Introduction}  A little-known but indispensable step in every proof of the local Langlands correspondence for $G = GL(n)$ over a $p$-adic field is the following {\it splitting theorem} of Henniart:
\begin{thm}\cite{He90}[Henniart]\label{gln}  Let $F_0$ be a non-archimedean local field, $n$ a positive integer, and $\pi_0$ a supercuspidal representation of $G(F_0)$.  There is a finite sequence of cyclic extensions $F_0 \subset F_1 \subset \dots \subset F_r$ such that, if we define $\pi_i$ by induction as the representation of $G(F_i)$ obtained as the cyclic base change from $F_{i-1}$ of $\pi_{i-1}$, then $\pi_r$ is not supercuspidal.
\end{thm}  

Since the local Langlands correspondence is a bijection 
$$\CL_F:  \CG(n,F) \rightarrow \CA(n,F)$$
between the set $\CA(n,F)$ of (equivalence classes of) $n$-dimensional (Frobenius semisimple, smooth) representations of the Weil-Deligne group of $F = F_i$ and the set $\CG(n,F)$ of (equivalence classes of) irreducible admissible representations of $G(F)$, in both cases with coefficients in a field of characteristic $0$, this theorem  is in fact an immediate consequence of the following properties:
\begin{itemize}
\item[(a)] The existence of the local Langlands correspondence;
\item[(b)] The property that $\pi$ is supercuspidal if and only if $\CL_F(\pi)$ is irreducible, and
\item[(c)] The property that cyclic base change from $\CG(n,F_{i-1})$ to $\CG(n,F_i)$ is taken under $\CL_F$ to restriction of Weil-Deligne representations from $F_{i-1}$ to $F_i$.
\end{itemize}

  Indeed, since $\CL_{F_0}(\pi_0)$ is necessarily trivial on an open subgroup of the absolute inertia group of $F_0$, after a finite sequence of cyclic extensions as in Theorem \ref{gln}, $\CL_{F_r}(\pi_r)$ is no longer irreducible, and thus $\pi_r$ cannot be supercuspidal.  
 One knows more:  we can find $r$ such that, writing $\CL_{F_r}(\pi_r) = (\sigma_r,N_r)$, where $\sigma_r$ is an $n$-dimensional representation of the Weil group $W_{F_r}$ of 
$F_r$ and $N$ is a nilpotent operator normalized by $\sigma_r(W_{F_r})$ and satisfying the Weil-Deligne relation with regard to $N$, we may assume that $\sigma_r$ is unramified.  Another property of the local Langlands correspondence then implies the following strengthening of Theorem \ref{gln}:

\begin{thm}\label{Iwahori}  We can find a sequence $F_{i-1} \subset F_i$ as in the notation of Theorem \ref{gln} such that $\pi_r$ contains a vector invariant by an Iwahori subgroup of $G(F_r)$.
\end{thm}

In fact, Henniart proves a stronger version of this theorem in \cite{He90}; the irreducibility of $\CL_{F_0}(\pi_0)$ easily implies that $N_r = 0$, and moreover that $\sigma_r$ is diagonal.  We keep the weaker version here because it is the most that can be expected when $GL(n)$ is replaced by a more general group $G$, and even then it appears one has to find a way to exclude cuspidal unipotent representations (on which more below).

Henniart derived these two theorems in \cite{He90} as a consequence of his {\it numerical correspondence} \cite{He88}.  This was the construction of bijections $\CL'_F:   \CG(n,F) \rightarrow \CA(n,F)$ that were not known to satisfy all the desirable properties of the local Langlands correspondence -- in fact, there were many such correspondences,  only one of which could be the right one -- but that did satisfy certain hereditary properties with respect to cyclic base change.  The proof of the numerical correspondence involved many steps, notably passage between $p$-adic fields and local fields of characteristic $p$, and the application of Laumon's local Fourier transform to the latter.  Most important, perhaps, is that the supercuspidal representations in $\CG(n,F)$ of fixed (minimal) conductor could be {\it counted}, because the Jacquet-Langlands correspondence places these representations in bijection with (a subset of the) irreducible representations of the multiplicative group of a central division algebra of dimension $n^2$ over $F$, and thus reduces the counting problem to the enumeration of representations of finite groups.  Then, by a long induction, Henniart is able to show that this number coincides with a similar number on the $\CA(n,F)$ side of the correspondence.   We wish to draw particular attention to this step because it is unique to reductive groups $G$ of type $A$; it has no analogue for other classes of groups.  

The term ``base change" in the above theorems is ambiguous.  When $F_0$ is of characteristic $0$ the base change of $\pi_{i-1}$ to $GL(n,F_i)$ asserts an explicit relation between the distribution character of $\pi_{i-1}$ and the twisted character of $\pi_i$ with respect to the cyclic group $Gal(F_i/F_{i-1})$; this was established by Arthur and Clozel by a combination of local and global means, using the full strength of the (twisted) trace formula for number fields.  These methods have not been completely developed over global fields of positive characteristic, and there are no complete references even for the case  of $GL(n)$, although it is known to experts.  Thus when $F_0$ is of positive characteristic cyclic base change for $\CG(n,\bullet)$ is simply defined by property (c) above, whereas (c) is proved directly for $p$-adic fields using formal properties of base change.  

The proofs of the full local Langlands correspondence for $GL(n)$ in \cite{LRS} (for $F$ of positive characteristic) or in \cite{HT01} \cite{He00} (for $p$-adic fields) all rely on the numerical correspondence of \cite{He88}.  The proofs for $p$-adic fields build more precisely on an inductive argument based on Henniart's Theorem \ref{gln}, which is used in an argument of Henniart (Theorem 3.3 of \cite{BHK}) to show that  maps from $\CG(n,F)$ to $\CA(n,F)$ that satisfy the formal properties of those constructed in \cite{H97, HT01} (using $p$-adic uniformization) or in \cite{He00} (using formal properties of representations of global Weil groups) are necessarily bijections.   \footnote{The proof in positive characteristic in \S VII.1 (b) of \cite{Laf02}, which is obtained as a consequence of L. Lafforgue's proof of the global Langlands correspondence, ends with a reference to Henniart's numerical correspondence.   However, it is not difficult to see that this is superfluous, because the properties of $L$-functions of pairs that Lafforgue uses already suffice to prove bijectivity.  
} 

Although, like the proofs in \cite{HT01,He00}, Peter Scholze's proof of the local Langlands correspondence for $GL(n)$ makes use of the construction of certain automorphic representations attached to representations of global Weil groups, his approach is based on rather different principles.  In particular, Scholze makes no use at all of Henniart's numerical correspondence, and does not need to pass between characteristic $0$ and characteristic $p$.  Nevertheless, his proof does use Theorem \ref{Iwahori}, but he proves it by a geometric argument, based on a novel analysis of nearby cycles on the tower of coverings of the Lubin-Tate formal moduli space of deformations of $1$-dimensional formal groups.  Like Henniart's numerical correspondence, this method has no obvious generalization to groups other than $GL(n)$.  While the representations of many $G(F)$ can  be realized on the cohomology of the $p$-adic period domains of Rapoport-Zink, and while there is work in progress of Fargues and Scholze that aims to construct a partial correspondence on the cohomology of spaces with actions by any $G(F)$ that can be constructed in the category of diamonds, the geometric properties that allow Scholze to analyze the nearby cycles seem specific to the Lubin-Tate tower,  and thus to $GL(n)$. 

The purpose of this note is to explain a third method to prove Theorems \ref{gln} and \ref{Iwahori} that do not depend on the uniquely favorable properties of $GL(n)$.   The method is based on a hypothetical combination of a local parametrization of representations of groups over local fields, analogous to that provided in positive characteristic by the work of Alain Genestier and Vincent Lafforgue, in the setting of the latter's global parametrization of automorphic representations, with analytic techniques based on the trace formula and on integral representations of $L$-functions.   

To illustrate the method, we begin with a proof of Theorem \ref{gln}  in characteristic $p$, assuming the extension to function fields of the results of \cite{AC} on base change.  
This is followed by a sketch of the argument deducing the local Langlands correspondence from Theorem \ref{gln}.  
This appears to qualify as a new proof, in that, instead of using the Lefschetz formula for the action of correspondences on moduli spaces of shtukas, and point counting, as in \cite{LRS,Laf02} (or the proofs for $p$-adic fields, all of which refer back to Shimura varieties), we use Vincent Lafforgue's geometric study of these moduli spaces.  However, the trace formula is used in the proofs of the main results of \cite{AC}.  Since the trace formula for function fields is still a work in progress, any proof based upon the methods of \cite{AC} must be considered conditional, although the specialists assure us that there is no essential difficulty.  \footnote{It is possible to provide a separate proof of the main local and global results of \cite{AC} that does not require the trace formula; instead, we could use the global Langlands correspondence for $GL(n)$, which is proved (without the trace formula) in \S 16 of \cite{Laf18}, as well as standard properties of Rankin-Selberg $L$-functions of $GL(n)\times GL(n)$.  However, as obserbved in the previous footnote, the proof of the global correspondence in \cite{Laf18} could easily be extended to prove the local correspondence as well, so we do not pursue this argument.}

 The first three sections are thus largely devoted to showing that the proofs indicated here, which are based on the constructions of \cite{GLa,Laf18}, and on the theory of  $L$-functions, are independent of the earlier constructions of \cite{LRS} and \cite{He88} as well as \cite{Laf02}.  

The next three sections introduce a weak notion of global solvable base change for general groups, and a still weaker notion of local solvable base change that is nevertheless sufficient for our applications. 
 Local solvable  base change takes an admissible irreducible representation of a reductive group $G$ over a local field $F$ to a (possibly infinite) set of irreducible admissible representations of $G(FÕ)$ for a finite solvable extension $FÕ$ of $F$.  
 (The notion is only defined for globalizable representations but can easily be generalized if one is willing to include all constituents of a parabolically induced representation in the set.).  
 Existence of global and local base change is conditional upon proof of at least a simple version of the stable twisted trace formula over function fields.  
 We assume this condition for the remainder of the introduction.

Define an {\it incorrigible representation} of $G(F)$ to be a supercuspidal representation $\pi$ of $G(F)$ such that, for any solvable $FÕ/F$, the set of base changes of $\pi$ to $G(FÕ)$ contains a supercuspidal member.  
Henniart's Theorem 1.1 is the statement that there are no incorrigible representations for $GL(n)$.  
For general $G$, Conjecture 4.4 asserts that no {\it pure} supercuspidal representation can be incorrigible, where purity is a property that can be defined consistently over local fields of positive characteristic using the local parametrization of Genestier and Lafforgue.  
In particular, the problematic cuspidal unipotent representations are (implicitly) excluded from consideration by the purity condition.

Section \ref{class} treats classical groups, using the doubling method of Piatetski-Shapiro and Rallis to control the poles of local $L$-functions.  We show that the refinement of this method in \cite{LR05} suffices to prove the analogue of Henniart's theorem \ref{gln} for pure supercuspidal representations.   
The theory of $L$-functions used here is native to classical groups, but the following section shows how to use automorphic descent, as in \cite{GRS}, to reduce questions about classical groups and $G_2$ to $GL(n)$.


In a final section we show how an analogue of Henniart's theorem can be derived from the Hiraga-Ichino-Ikeda conjecture on formal degrees \cite{HII08}.  


The first draft of \cite{BHKT} was optimistic about the possibility of using cyclic descent, in the setting of Vincent Lafforgue's theory of Langlands parameters
for groups over function fields, to prove new results about the local Langlands correspondence for reductive groups over local fields of positive
characteristic.  Our optimism was premature, however:  it was pointed out by J.-P. Labesse, C. Moeglin, and J.-L. Waldspurger, not only that the theory of the
stable twisted trace formula remains to be developed over function fields -- that's already a serious obstacle -- but also that, even supposing that such a
theory were available, its methods do not suffice to control either local or global multiplicities for general groups (essentially for exceptional groups).
This means that the twisted traces we were hoping to use to carry out cyclic descent could very easily be identically zero.  The concept of incorrigible representations
was introduced in order to understand the local obstruction to cyclic descent.  I believe (but I have not checked) that the truth of Conjecture 4.4 would suffice to eliminate this obstruction.  I leave it to the reader to define the analogous notion of incorrigible cuspidal automorphic representations.

Thus this paper owes its existence in large part to my exchanges with Labesse, Moeglin, and Waldspurger, and I thank them for their generosity and patience.
In preparing this paper I have also benefited from helpful comments and suggestions by G. Henniart and L. Lomel\'i.  And I am  grateful once again to J.-L. Waldspurger, who pointed out numerous imprecisions in an earlier version of the manuscript, and who was nevertheless generous enough to read and suggest improvements to the first arXiv version as well.   Finally, this paper is an outgrowth of my  collaboration \cite{BHKT} with G. B\"ockle, C. Khare, and J. Thorne.  Athough they finally chose not to be co-authors of this article, I am grateful for their numerous suggestions; their ideas and perspectives  thoroughly influenced the shape of the article. 



\section{The splitting theorem for $GL(n)$}

Let $X$ be a smooth projective curve over the finite field $k$ of characteristic $p$.  Let $K = k(X)$ denote its function field.  Let $\mathbf{G}_n$ be the group $GL(n)$, viewed as a reductive algebraic group over $K$.  We let $C$ be an algebraically closed coefficient field of characteristic zero; it could be $\C$ or it could be $\Qlb$ with $\ell \neq p$.  All automorphic representations of the groups we will consider,  as well as all irreducible representations of groups over local fields, will have coefficients in $C$.

\begin{hyp}\label{ACBC}  The results of \cite{AC} are valid over $K$.   In particular, let $K'/K$ be a finite extension, and let $K''/K'$ be a cyclic extension of prime degree $q$.  Let $\Pi'$ be a cuspidal automorphic representation of $\mathbf{G}_{n,K'}$.  Then there is an automorphic representation $\Pi'' = BC_{K''/K'}(\Pi')$ of $\mathbf{G}_{n,K''}$ with the following properties:

\begin{itemize}
\item[(a)]  Let $v \in |X_{K'}|$ be a place of $K'$ and suppose $\Pi'_v$ is unramified.  Let $w$ be a place of $K''$ dividing $v$.  Then $\Pi''_w$ is the unramified principal series representation of $GL(n,K''_w)$ obtained by unramified base change from $\Pi'_v$.  In other words, if $K^{\prime,un}_v$, resp. $K^{\prime\prime,un}_w$, is the maximal unramified extension of $K'_v$, resp. of $K''_w$, and if $\sigma_v:  Gal(K^{\prime,un}_v/K'_v) \ra GL(n,C)$ is the Satake parameter of $\Pi'_v$, then $\Pi''_w$ is the unramified principal series representation with Satake parameter $\sigma_v ~|_{Gal(K^{\prime\prime,un}_w/K''_w)}$.

\item[(b)]  Let $\alpha:  Gal(K''/K') \ra C^{\times}$ be a non-trivial character.  Then $\Pi''$ is cuspidal if and only if $\Pi'\otimes \alpha \not\equiv \Pi'$.  If $\Pi'\otimes \alpha \equiv \Pi'$ then $q$ divides $n$; letting $d = n/q$, there is a cuspidal automorphic representation $\Pi_0$ of $GL(d)_{K''}$ such that 
$$\Pi'' \isoarrow \Pi_0 \boxplus \Pi_0^{\tau} \boxplus \dots \boxplus \Pi_0^{\tau^{q-1}},$$
where $\tau \in Gal(K''/K')$ is any non-trivial element, and $\boxplus$ is the Langlands sum.
\end{itemize}
\end{hyp}

\begin{hyp}\label{ACBCloc}  Let  $k$  be a finite field of characteristic $p$ and let $F$ be the local field $k((T))$.  The results of \cite[\S 1]{AC} are valid over $F$.   In particular, let $F'/F$ be a cyclic Galois extension of prime degree $q$.  

\begin{itemize}
\item[(a)]  Let $\pi$ be a supercuspidal representation of $GL(n,F)$.   Then there is a representation $\pi' = BC_{F'/F}(\pi)$ of $GL(n,F')$ with the following property.  Suppose $K'/K$ is a cyclic Galois extension of degree $q$ of global function fields,  $v$ is a place of $K$ that has a unique extension $w$ to $K'$, and there are isomorphisms $K_v \isoarrow F$, $K'_w \isoarrow F'$, so that $Gal(K'/K)$ equals the decomposition group of $w$ over $v$.  Let $\Pi$ be any cuspidal automorphic representation of $\mathbf{G}_{n,K}$ such that $\Pi_v \isoarrow \pi$, and let $\Pi'$ denote the base change of $\Pi$ to $\mathbf{G}_{n,K'}$ guaranteed by Theorem \ref{ACBC}.  Then 
$\Pi'_w \isoarrow \pi'$.  In particular, the local component of $\Pi'$ at $w$ depends only on $\Pi_v$.

\item[(b)]  Let $\alpha:  Gal(F'/F) \ra C^\times$ be a non-trivial character.  Then $\pi'$ is supercuspidal if and only if $\pi\otimes \alpha \not\equiv \pi$.  If $\pi\otimes \alpha \equiv \pi$ then $q$ divides $n$; letting $d = n/q$, there is a supercuspidal  representation $\pi_0$ of $GL(d,F')$ such that 
$$\pi' \isoarrow \pi_0 \boxplus \pi_0^{\tau} \boxplus \dots \boxplus \pi_0^{\tau^{q-1}},$$
where $\tau \in Gal(F'/F)$ is any non-trivial element, and $\boxplus$ is the Langlands sum.  Conversely, if $\pi_0$ is a supercuspidal representation of $GL(d,F')$ such that $\pi_0 \not\equiv \pi_0^{\tau}$ for any non trivial $\tau \in Gal(F'/F)$, then the Langlands sum above descends to a supercuspidal representation of $GL(n,F)$.

\item[(c)]  Suppose $\pi'$ is supercuspidal in (b) above.  Suppose $\sigma$ is an irreducible admissible representation of $GL(n,F)$ such that $BC_{F'/F}(\sigma) \isoarrow \pi'$.  Then there is an integer $a$ such that $\sigma \isoarrow \pi \otimes \alpha^a\circ \det$, where $\alpha$ is as in (b).
\end{itemize}



\end{hyp}

As noted in the footnotes to the introduction, Hypotheses \ref{ACBC} and \ref{ACBCloc} can be proved without use of the trace formula, but that would defeat the
purpose of this paper.

Other results of \cite{AC} will be used in the course of the discussion, and they will be listed at the end of the following section.  We want to avoid using the full strength of \cite{AC}, in fact.
The strong multiplicity one theorem for $GL(n)$, and the more general classification theorem of Jacquet and Shalika, guarantees that $\Pi''$ is uniquely determined by (a) of Hypothesis \ref{ACBC}.  This implies in particular that, for any place $v \in |X_{K'}|$, there is a map associating $\Pi''_w$ to $\Pi'_v$, whether or not $\Pi'_v$ is ramified.  It is by no means obvious that $\Pi''_w$ is independent of the global representation $\Pi'$; Arthur and Clozel prove that this is the case, but we prefer not to include this in the hypotheses, because the situation for more general groups is more complicated.  Thus we introduce the following ad hoc definition:

\begin{defn}\label{BCcontingent} Let $F$ be a local field of characteristic $p > 0$, and let $\pi$ be an irreducible representation of $GL(n,F)$.  Let $K' = k(X')$ be a global function field and let $v \in |X'|$ be a place of $K'$ that admits an isomorphism $K'_v \isoarrow F$.   Let $\Pi'$ be a cuspidal automorphic representation of $GL(n)_{K'}$ such that $\Pi'_v \isoarrow \pi$.  Let $F'/F$ be a cyclic extension of prime degree $q$, and let $K''/K'$ be an extension of degree $q$ in which $v$ is inert, and such that, letting $w$ denote the prime of $K''$ dividing $v$, we have an isomorphism $K''_w \isoarrow F'$.  Define
$$BC^{\Pi',K''}_{F'/F}(\pi) = BC_{K''/K'}(\Pi')_w$$
to be the local component at $w$ of $BC_{K''/K'}(\Pi')$.
\end{defn}

Note that this only defines local base change for representations that occur as local components of cuspidal automorphic representations.  This includes supercuspidal representations with central characters of finite order (see Theorem \ref{glob} below), and this suffices for our purpose.  

\subsection*{Parametrization}\label{parametrization}
In order to define the version of local base change that we need, we start with supercuspidal representations.   For the moment we work in the generality of \cite{Laf18}.  Thus let $K \subset K' \subset K''$ be as above, and let $\mathbf{G}$ be any connected reductive algebraic group over $K$.  We fix a point $v \in |X_{K'}|$ and let $F' = K'_v$ denote the corresponding local field, $G = \mathbf{G}(F)$.  Let $\pi_v$ be a supercuspidal representation of $G(F')$ with  central character of finite order.  

\begin{thm}[Gan-Lomel\'i]\cite{GLo}\label{glob}  There a cuspidal automorphic representation $\pi$ of $\mathbf{G}_{K'}$ with central character of finite order, such that the local component $\pi_v$ is the given representation.
\end{thm}  

Gan and Lomel\'i prove a much  more refined result, with strong restrictions on the ramification of $\pi$ away from $v$, but this simple version will suffice for our purposes.   Before we return to the case of $GL(n)$, we quote the main results of \cite{Laf18} and \cite{GLa}.  For this, let $F$ be any of the fields $K, K', K'', K'_v, K''_w$, and let ${}^LG = \hat{G} \rtimes W_F$ be the $L$-group over $F$, where $\hat{G}$ is the Langlands dual group of $G$, with coefficients in the field $C$, and $W_F$ is the (global or local) Weil group.  

If $F$ is a local field, we let $G = \mathbf{G}(F)$.  If $F'/F$ is any extension,  we let $\CG(G,F')$ denote the set of irreducible admissible representations of $G(F')$.  We let $\CA(G,F')$ denote the set of {\it admissible Langlands parameters}
$$\phi:  WD_{F'} \ra {}^LG$$
with the usual properties; in particular, the restriction of $\phi$ to $W_{F'}$, composed with the tautological map ${}^LG \ra W_F$, is the identity map from $W_{F'}$ to the subgroup $W_{F'} \subset W_F$.   Let $\CA(G,F')^{ss}$ denote the set of equivalence classes of (Frobenius semisimple, smooth) homomorphisms of the Weil group of $F'$ to $ {}^LG$.  There is a natural map $\CA(G,F') \ra \CA(G,F')^{ss}$ given by forgetting the image of the nilpotent operator $N$.

If $F$ is one of the global fields $K, K', K''$, we let $\CG_0(\mathbf{G}_F)$ denote the set of cuspidal automorphic representations of $\mathbf{G}_F$ with central character of finite order.  We let 
$\CA^{ss}(\mathbf{G}_F)$ denote the set of equivalence classes of compatible families of completely reducible $\ell$-adic representations, for $\ell \neq p$:
$$\rho_\ell:  Gal(\bar{F}/F) \ra {}^LG(\Qlb).$$
The term {\it completely reducible} is understood to mean that if $\rho_\ell(Gal(\bar{F}/F))\cap \hat{G}(\Qlb)$ is contained in a parabolic subgroup $P \subset \hat{G}(\Qlb)$, then it is contained in a Levi subgroup of $P$.  

If $\nu:   G \ra \Gm$ is an algebraic character, with $\Gm$ here designating the split $1$-dimensional torus over $F$, the theory of $L$-groups provides a dual character $\hat{\nu}:  \Gm \ra \hat{G} \subset {}^LG$.  If $Z$ is the center of $G$, and if $c:  \Gm \ra Z \subset G$ is a homomorphism, then the theory of $L$-groups provides an algebraic character ${}^Lc:  {}^LG \ra \Gm$.

\begin{thm}\label{paramVL} 
 (i) \cite{Laf18}[Th\'eor\`eme 0.1]    There is a map
$$\CL^{ss}_{K} = \CL^{ss}_{\mathbf{G},K}:  \CG_0(\mathbf{G}_K) \ra \CA^{ss}(\mathbf{G}_K)$$
with the following property:  if $v$ is a place of $K$ and $\Pi \in \CG_0(\mathbf{G}_K)$ is a cuspidal automorphic representation such that $\Pi_v$ is unramified, then $\CL^{ss}_K(\Pi)$ is unramified at $v$, and
$\CL^{ss}_K(\Pi) ~|_{W_{K_v}}$ is the Satake parameter of $\Pi_v$.\footnote{Here and below we will mainly refer to the restriction of a global Galois parameter to the local Weil group, rather than to the local Galois group, because the unramified Langlands correspondence relates spherical representations to unramified homomorphisms of the local Weil group to the $L$-group.   But the difference is inessential.}

(ii)  Suppose  $\nu:   \mathbf{G} \ra \Gm$ is an algebraic character.  Suppose $$\chi:  \ad_K^{\times}/K^\times  \ra C^\times = GL(1,C)$$ is any continuous character of finite order.  For any $\Pi \in  \CG_0(\mathbf{G}_K)$, let $\Pi\otimes \chi \circ \nu$ denote the twist of $\Pi$ by the character $\chi\circ\nu$ of $\mathbf{G}(\ad)$. Then 
$$\CL^{ss}(\Pi\otimes \chi\circ\nu) = \CL^{ss}(\Pi)\cdot \hat{\nu}\hat{\chi}$$
where $\hat{\nu}$ is as above and where $\hat{\chi}: Gal(\bar{K}/K)^{ab} \ra GL(1,C)$ is the character corresponding to $\chi$ by local class field theory. 

(iii) \cite{GLa}[Th\'eor\`eme 0.1]  Let  $v$ be a place of of $K$ and let $F = K_v$.  Then the semisimplification of the restriction of $\CL^{ss}_K(\Pi)$ to $W_{K_v}$ depends only on $F$ and $\Pi_v$ and not on the rest of the automorphic representation $\Pi$, nor on the global field $K$.  

\end{thm}

\begin{proof}  Point (i) is the main result of \cite{Laf18}, and point (iii) is the main result of \cite{GLa}. Point (ii) follows from point (i):  it is true locally at almost all places, by the compatibility with the unramified local correspondence, and thus it is true everywhere by Chebotarev density. 
\end{proof}

The following corollary is implicit in \cite{GLa}.

\begin{cor}\label{compatibilities} Let $F$ be a local field of characteristic $p$.  There is a map 
$$\CL^{ss}_{F}: \CG(G,F) \ra  \CA(G,F)^{ss}$$
with the following properties.  

(i)  If $\pi \in \CG(G,F)$ is an unramified principal series representation, then $\CL^{ss}_{F}(\pi)$ is its Satake parameter.

(ii) More generally, $\CL^{ss}_{F}$ is compatible with parabolic induction, in the following sense. Suppose $M \subset G$ is the Levi subgroup of an $F$-rational parabolic subgroup $P$ of $G$, and let $i_M:  {}^LM \ra {}^LG$ be the corresponding morphism.  Let $\sigma \in \CG(M,F)$ and let $\pi$ be an irreducible constituent of $Ind_{P(F)}^{G(F)}\sigma$ (normalized induction).  Then 
 $$\CL^{ss}_{G,F}(\pi) = i_M(\CL^{ss}_{M,F}(\sigma)).$$
 
 (iii)  The map $\CL^{ss}_{F'}$ is compatible with twisting by characters, in the following sense.  Let $\nu:   G \ra \Gm$ be an algebraic character.  Suppose $\chi:  F^{\prime,\times} \ra C^{\times}$ is a continuous character, and $\hat{\chi}:  W_{F'} \ra GL(1,C)$ is the representation corresponding to $\chi$ by local class field theory.  Then for any $\pi \in \CG(G,F')$, we have
 $$\CL^{ss}_{F'}(\pi \otimes \chi\circ \nu) = \CL^{ss}_{F'}(\pi)\cdot \hat{\nu}_{\hat{\chi}}:  W_{F'} \ra {}^LG.$$
 
 (iv)  The map $\CL^{ss}_{F}$ is compatible with central  characters, in the following sense.  Let $c:  \Gm \ra Z \subset G$ and ${}^Lc$ be as in the discussion above.  Suppose 
 $\pi \in \CG(G,F)$ has central character $\xi_\pi$, and let $\xi_{c,\pi} = \xi_{\pi}\circ c:  Z(F) \ra C^{\times}$.  Then ${}^Lc \circ \CL^{ss}_{F}(\pi)$ is the character attached to $\xi_{c,\pi}$ by local class field theory.
 
 (v)  Suppose $K = k(X)$ is a global function field and $v \in |X|$ is a place such that $F \isoarrow K_v$.  Then $\CL^{ss}_{F}$ is compatible with the map $\CL^{ss}_{K}$:  if $\Pi$ is any cuspidal automorphic representation of $\mathbf{G}_K$ with finite central character, then
 $\CL^{ss}_F(\Pi_v)$ is equivalent to the semisimplification of  $\CL^{ss}_K(\Pi) ~|_{W_{F_v}}$.
\end{cor}  

\begin{proof}  Suppose $\pi \in \CG(G,F)$ is a supercuspidal representation with central character of finite order.  Let $K = k(X)$ be a global function field with a place $v$ such that $F \isoarrow K_v$.  By Theorem \ref{glob} there is a cuspidal automorphic representation $\Pi$ with central character of finite order, such that $\Pi_v \equiv \pi$.  Then we define
$\CL^{ss}_F(\pi)$ to be the semisimplification of  $\CL^{ss}_K(\Pi) ~|_{W_{F_v}}$, as required in point (v).  It follows from point (iii) of Theorem \ref{paramVL} that this is well defined.

More generally, if $\pi \in \CG(G,F)$ is any supercuspidal representation, we may find a $C$-valued character $\chi$ of $G(F)$ such that $\pi\otimes \chi$ has central character of finite order.  We apply the previous step to $\pi\otimes \chi$ and define
$$\CL^{ss}_F(\pi) = \CL^{ss}_F(\pi\otimes \chi)\otimes \hat{\chi}^{-1}.$$
That this is well-defined follows from point (ii) of Theorem \ref{paramVL}.

Now suppose $\pi$ is an irreducible constituent of $Ind_{P(F)}^{G(F)}\sigma$ and define $\CL^{ss}_{G,F}(\pi)$ by (ii).  We need to show that this definition is compatible with (v); but this follows from the final assertion of Theorem 0.1 of \cite{GLa}.  

Finally, points (iii) and (iv) are true at all unramified places, by construction; hence they are true everywhere by Chebotarev density.
\end{proof}

The following Theorem contains the main local-global compatibility property of base change.

\begin{thm}\cite{Laf18,GLa}\label{chtouca}  Let $F$ be a local field of characteristic $p$, and let $F'/F$ be a cyclic extension of prime degree $q$.  Let $\pi$ be an irreducible representation of $GL(n,F)$ and let $K'$, $v$, $\Pi'$, and $K''$ be as in Definition \ref{BCcontingent}.  

Then
$$\CL^{ss}_{F'}(BC^{\Pi',K''}_{F'/F}(\pi)) = \CL^{ss}_F(\pi)~|_{W_{F'}}.$$
\end{thm} 

\begin{proof}  If $\pi$ is spherical then this follows from (i) of Theorem \ref{paramVL}.   The general case then follows from Chebotarev's density theorem.
\end{proof}

\subsection*{Proof of Theorem \ref{gln}}

Now we return to the case of $GL(n)$.  Because Definition \ref{BCcontingent} of base change is not purely local, we need to reformulate Theorem \ref{gln} to take this into account. 

\begin{thm}\label{gln1}  Let $F_0$ be a non-archimedean local field of positive characteristic $p$, $n$ a positive integer, and $\pi_0$ a supercuspidal representation of $GL(n,F_0)$.   Choose  a global function field $K_0$\footnote{In what follows we are using Theorem 5 of Chapter X of \cite{AT}.  If we also want to work over number fields we would have to take care to avoid the Grunwald-Wang obstruction at primes dividing $2$} with a place $v_0$ such that $K_{0,v_0} \isoarrow F_0$, and a cuspidal automorphic representation $\Pi_0$ of $GL(n,\ad_{K_0})$, as in Theorem \ref{glob}, such that
$\Pi_{0,v_0} \isoarrow \pi_0$.  
There is a finite sequence of cyclic extensions of prime degree, $F_0 \subset F_1 \subset \dots \subset F_r$, with the following property.  Let $K_0 \subset K_1 \subset \dots \subset K_r$ be any sequence of cyclic extensions, with $v_0$ inert in $K_r$, such that, for each $i$, $K_{i,v_i} \equiv F_i$, where $v_i$ is the prime of $K_i$ above $v$.  Define $\Pi_i$ inductively as $BC_{K_i/K_{i-1}}(\Pi_{i-1})$, and let  $\pi_i = \Pi_{i,v_i}$.  Then  $\pi_r$ is not supercuspidal.
\end{thm}  

\begin{proof}  The proof is based on properties of the Godement-Jacquet $L$-function for $GL(n)$.   We fix a prime $\ell \neq p$ and view  $\CA^{ss}(\mathbf{G}_{K_i})$ as sets of representations of the respective Galois groups on $n$-dimensional vector spaces over $\Qlb$.  Let $F_r/F_0$ be a Galois extension such that 
$\CL^{ss}_{F_0}(\pi_0)$ restricts to a (semisimple) unramified extension of $Gal(\bar{F}_0/F_r)$.  Since Galois groups of local fields are solvable, we can find a finite sequence of cyclic extensions of prime degree $F_i/F_{i-1}$, $i = 1, \dots, r$.  Now let $K_0$,  $\Pi_0$, and $K_0 \subset K_1 \subset \dots \subset K_r$, be as in the statement of the theorem.

It follows from (v) of Corollary \ref{compatibilities} that the semisimplification of the restriction of 
$\CL^{ss}_{K_r}(\Pi_r)$ to $Gal(\bar{K_{r,v_r}}/K_{r,v_r})$ is unramified.  Note that $\CL^{ss}_{K_r}(\Pi_r)$ is a semisimple representation of the global Galois group, but its restriction to the local Galois group need not be semisimple.  Nevertheless, by Grothendieck's monodromy theorem, it follows that 
\begin{itemize}
\item The image of $\CL^{ss}_{K_r}(\Pi_r) ~|_{Gal(\bar{K}_{r,v_r}/K_{r,v_r})}$, viewed as a subgroup of $GL(n,\Qlb)$, fixes a line in $\Qlb^n$, and acts on this line by an unramified character, say $\alpha_r$.
\end{itemize}
This implies that
\begin{equation}\label{pole1}
\text{ The local Euler factor $L_{v_r}(s,\CL^{ss}_{K_r}(\Pi_r)\otimes \alpha_r^{-1})$ has a pole at $s = 0$.}
\end{equation}
Now fix a global Hecke character $\beta$ of $GL(1)_{K_r}$ and consider the Godement-Jacquet $L$-function 
$$L(s,\Pi_r,\beta) = \prod_w L_w(s,\Pi_{r,w},\beta_w).$$
There is a finite set $S$ of places $w$ of $K_r$ such that, for $w \notin S$ $\Pi_{r,w}$ is unramified, and it follows from (i) of Theorem \ref{paramVL} that 
\begin{equation}\label{localGJ}  L_w(s,\Pi_{r,w},\beta_w) = L_w(s,\CL^{ss}_{K_r}(\Pi_r)\otimes \beta_w).
\end{equation}

We assume $S$ contains $v_r$.
Moreover, for any Hecke character $\beta$, viewed alternately as an automorphic representation of $GL(1)$ or as an $\ell$-adic Galois character, $L(s,\Pi_r,\beta)$ and the Artin $L$-function $L(s,\CL^{ss}_{K_r}(\Pi_r)\otimes \beta)$ satisfy functional equations
$$L(s,\Pi_r,\beta) = \varepsilon(s,\Pi_r,\beta)\cdot L(1-s,\Pi_r^{\vee},\beta^{-1}); $$
$$L(s,\CL^{ss}_{K_r}(\Pi_r)\otimes \beta) =  \varepsilon(s,\CL^{ss}_{K_r}(\Pi_r)\otimes \beta)\cdot L(1-s,\CL^{ss}_{K_r}(\Pi_r)^{\vee}\otimes \beta^{-1}).$$
The local factors of the automorphic and Galois functional equations are equal outside $S$ by \eqref{localGJ}.  Moreover, by stability of $\gamma$-factors \cite{JS85, De73}, if $\beta$ is sufficiently ramified at all places in $S$ other than $v_r$, the local factors at such places are also equal.  It follows as in the usual argument that 
\begin{equation}\label{equals1} L_{v_r}(s,\Pi_{r,v_r},\beta_{v_r}) = L_{v_r}(s,\CL^{ss}_{K_r}(\Pi_r)\otimes \beta_{v_r}).
\end{equation}

Now we can choose a global character $\beta$ such that $\beta_{v_r} = \alpha_r^{-1}$ and such that $\beta$ is sufficiently ramified at all places in $S$ other than $v_r$.  By \eqref{pole1} and \eqref{equals1}, it follows that $L_{v_r}(s,\Pi_{r,v_r},\beta_{v_r})$ has a pole at $s = 0$.  But if $\tau$  is a supercuspidal representation of $GL(n,F_r)$, then the Godement-Jacquet local Euler factor $L(s,\tau,\chi) = 1$ for any character $\chi$.  Thus $\Pi_{r,v_r}$ is not supercuspidal.
\end{proof}

\section{Proof of the local Langlands correspondence}\label{localGLn}

Admitting Hypothesis \ref{ACBC}, Theorem \ref{gln1} can be used as the starting point of an apparently new proof (but see Footnotes 1 and 2) of the local Langlands correspondence for $GL(n)$ in the equal characteristic case.  
Indeed, the inductive arguments of section 12 of \cite{Sch13} are based entirely on this base change argument and local class field theory -- Scholze obtains the bijection after reproving Henniart's Theorem \ref{gln}, while Henniart obtained the result in the other direction, starting from his numerical correspondence.   
That the correspondence preserves $L$ and $\varepsilon$ factors of pairs is automatic by global arguments over function fields, specifically the fact that the $L$-functions of Galois representations are known to satisfy the expected functional equations; this has already been used to prove Theorem \ref{gln1}.  

	The sense in which this proof is actually {\it new} needs to be spelled out, of course.  The result has been known since \cite{LRS} and can also be derived from \cite{Laf02}, and many of the intermediate arguments used to deduce the proof are the same in all cases.    
	
	In order to convince the reader that it is possible to prove the local Langlands correspondence for $GL(n)$ without using trace formulas to count points on moduli spaces, we indicate the steps of the proof, starting from Theorem \ref{gln1}.  This will be a line-by-line review of Scholze's proof in \cite{Sch13}.    We write $\CL^{ss}_{n,F}$ for the local parametrization denoted $\CL^{ss}_{F}$ in the previous section.

	\begin{step}\label{step1}  As $n$ and $F$ vary, the parametrizations $\pi \mapsto \CL^{ss}_{n,F}(\pi)$ define a functorial extension of class field theory, in the sense of \cite{Sch13}, Theorem 12.1.	
	\end{step}
	
	\begin{proof}  There are five conditions to check.  
	\begin{enumerate}
	\item[(i)]  When $\CL_{1,F} = \CL^{ss}_{1,F}$ coincides with local class field theory. 
	\medskip
	
	This is a special case of (iii) of Theorem \ref{compatibilities}  
	
	\item[(ii)] The map $\pi \mapsto \CL^{ss}_{n,F}(\pi)$ commutes with parabolic induction in the sense that, if $\pi$ is a subquotient of a representation $I(\pi_1\otimes \dots \otimes \pi_r)$ induced from a parabolic subgroup with Levi factor $\prod_{i = 1}^r GL(n_i)$, then
	$$\CL^{ss}_{n,F}(\pi) = \oplus_{i = 1}^r \CL^{ss}_{n_i,F}(\pi_i).$$
	\medskip	
	This is the final point of \cite{GLa}, Theorem 0.1.
	
	\item[(iii)]  If $\chi \in \CA(1,F)$ then $\CL^{ss}_{n,F}(\pi\otimes \chi\circ \det) = \CL^{ss}_{n,F}(\pi)\otimes \CL_{1,F}(\chi)$.
	\medskip	
	Since this is true for unramified representations, it follows in general by the Chebotarev density argument already used.
	
	\item[(iv)]  Compatibility with base change:  if $\pi \in \CA(n,F)$ is a supercuspidal representation and $F'/F$ is a cyclic extension of prime degree, then (for any $K, K'', \Pi'$ as in Definition \ref{BCcontingent}) 
$$\CL^{ss}_{F'}(BC^{\Pi',K''}_{F'/F}(\pi)) = \CL^{ss}_F(\pi)~|_{W_{F'}}.$$
	\medskip
		
	This is Theorem \ref{chtouca}.  
	
	\item[(v)]  If $\CL^{ss}_{n,F}(\pi)$ is unramified then $\pi$ has an Iwahori fixed vector.
		\medskip
		
	In particular, the proof of Theorem \ref{Iwahori} for $GL(n)$ is a step in the proof of the local Langlands correspondence.   Condition (v) follows from Theorem \ref{gln1}, but the proof is not immediate and requires a separate step.	
	\end{enumerate}
	
	\end{proof}
	
	\begin{step}\label{step2}  Proof of condition (v) of Step \ref{step1}	
	\end{step}
	
	\begin{proof}  Suppose $n = \sum_{i = 1}^r n_i$, $P \subset GL(n)$ a parabolic subgroup with Levi factor $\prod_i GL(n_i)$, and $\pi$ is a constituent of 
	$Ind_{P(F)}^{GL(n,F)} \otimes_{i = 1}^r \sigma_i$, where $\sigma_i$ is a supercuspidal representation of $GL(n_i,F)$ such that  $\CL^{ss}_{n,F}(\pi)$ is unramified.  It follows from condition (ii) of Step \ref{step1} that $\CL^{ss}_{n_i,F}(\sigma_i)$ is unramified.  On the other hand, if each $\sigma_i$ has an Iwahori-fixed vector then every irreducible constituent of $Ind_{P(F)}^{GL(n,F)} \otimes_{i = 1}^r \sigma_i$ has an Iwahori fixed vector.  It follows that it suffices to treat the case where $\pi$ is supercuspidal and $\CL^{ss}_{n,F}(\pi)$ is unramified, which we assume henceforward; the conclusion will be  that $n = 1$. 
	
	By Theorem \ref{gln1} we know that there is a sequence of cyclic extensions $F_0 \subset F_1 \subset \dots \subset F_r$ such that the base change $\pi_r$ of $\pi$ to $F_r$ has an Iwahori fixed vector.  By induction on $r$ and on $n$ we are thus reduced to verifying the following statement:  Suppose $\pi \in \CA(n,F)$ is supercuspidal and $F'/F$ is a cyclic extension of prime degree such that $\Pi = BC_{F'/F}(\pi)$ has an Iwahori fixed vector.  Suppose moreover that $\CL^{ss}_{n,F}(\pi)$ is unramified.  Then $\pi$ is an unramified principal series representation (and in particular $n = 1$).  
	
	We follow the argument in Theorem 12.3 of \cite{Sch13} (which goes back to Henniart).    By Step \ref{Arthur-Clozel} below, it follows that 
	$$\Pi = \boxplus_{j = 0}^{n-1}  {}^{\tau^j}(\chi)$$
	where $\chi$ is a character of $GL(1,F')$ such that ${}^{\tau}(\chi) \neq \chi$.  However, by compatibility with local class field theory and parabolic induction (conditions (i) and (ii) of Step \ref{step1}) we know  that $\chi$ is unramified.  Since $Gal(F'/F)$ acts trivially on unramified characters, $\pi$ could not have been supercuspidal, and indeed it follows that $\pi$ is necessarily a principal series representation.  Since $\CL^{ss}_{n,F}(\pi)$ is unramified,  it follows again from conditions (i) and (ii) of Step \ref{step1} that $\pi$ is an unramified principal series representation.	
	\end{proof}

\begin{step}\label{Arthur-Clozel}  Let $\pi \in \CA(n,F)$ be supercuspidal and suppose $F'/F$ is a cyclic extension of prime degree such that 
$\Pi = BC_{F'/F}(\pi)$ is not supercuspidal.  Then there is a divisor $m$ of $n$, with $n = md$, and a supercuspidal representation $\Pi_1$ of $GL(m,F')$, such that
$$\Pi = \boxplus_{j = 0}^{d-1}  {}^{\tau^j}(\Pi_1)$$
where $\tau$ is a generator of $Gal(F'/F)$, and such that ${}^{\tau}(\Pi_1) \neq \Pi_1$.

Conversely, if $\Pi_i$ is a supercuspidal representation of $GL(m,F')$ such that ${}^{\tau}(\Pi_1) \neq \Pi_1$, then $\Pi$, defined as above, descends to a discrete series representation $\pi$ of $GL(n,F)$.  
\end{step}

\begin{proof}  The first statement is Lemma 6.10 of Chapter I of Arthur-Clozel, \cite{AC}.   The proof of this result, which appears in the local part of \cite{AC} depends on the existence of cyclic base change, which we have admitted for function fields (Hypothesis \ref{ACBC}).  However, a close look at the proof of Theorem 6.2 of Chapter I of \cite{AC}, on which Lemma 6.10 depends, indicates that it suffices to use the Deligne-Kazhdan simple trace formula, and this is already available for function fields.  

The second statement follows from Theorem 6.2 (b) of Chapter I of \cite{AC}.  There it is proved that there is at least one $\pi \in \CA(n,F)$ whose base change to $GL(n,F)$ is $\Pi$.  Moreover, the hypothesis ${}^{\tau}(\Pi_1) \neq \Pi_1$  implies that $\Pi_1$ is $\sigma$-discrete, and an examination of the proof of (b) on p. 53 of \cite{AC} shows that $\pi$ must belong to the discrete series.   
\end{proof}
 
\begin{remark}  The proof of Lemma 6.10 also depends on the Bernstein-Zelevinsky classification of representations of $GL(n,F)$, and thus on the complete local theory for $GL(n)$.  The normalization of the twisted trace formula in Chapter I, \S 2 of \cite{AC} makes use of the Whittaker model.  The arguments here thus do not extend to groups other than $GL(n)$; however, at no point have we used the Lefschetz formula to study the cohomology of moduli spaces.
\end{remark}

\begin{step}\label{Scholze}[Local Langlands bijection]   The map $\pi \mapsto \CL^{ss}_{n,F}(\pi)$ restricts to a bijection between supercuspidal representations of $GL(n,F)$ and irreducible $n$-dimensional representations of the Weil group $W_F$.
\end{step}

\begin{proof}  At this point we are ready to argue as in the proof of Theorem 12.3 of \cite{Sch13}.  Scholze's argument is a formal consequence of Steps \ref{step1} and \ref{Arthur-Clozel}.   In particular, Scholze's induction argument (a version of which goes back to \cite{He88}) implies that the representation $\Pi_{m-1}$ that appears in the middle of his proof of Theorem 12.3, and that is guaranteed to belong to the discrete series by Step \ref{Arthur-Clozel}, is necessarily supercuspidal.  Moreover, the surjectivity of the parametrization is proved by induction on the degree of the representation, as follows:   If $\rho$ is an irreducible $n$-dimensional representation of $W_F$, find a solvable extension $F'$ over which it breaks up as the sum of $d$ $m$-dimensional irreducible representations.  By induction we may assume $F'/F$ is cyclic of prime degree $d$.  Each summand is the image of a supercuspidal representation of $GL(m,F)$, whose Langlands sum is thus invariant under $Gal(F'/F)$; then we apply (b) of Theorem \ref{ACBCloc} to conclude.  

\end{proof}

\begin{step}\label{localfactors}[Compatibility of local factors]  The bijection of Step \ref{Scholze} preserves $L$ and $\varepsilon$ factors of pairs.
\end{step}

\begin{proof}  This is simpler in positive characteristic than for $p$-adic fields, because the $L$-function of a representation of the Galois group of a function field is always known to have a meromorphic continuation with functional equation.  Thus the stability of local $\gamma$-factors can be used to verify the compatibility of local factors, as in the  proof of Theorem \ref{gln1}.
\end{proof}

\section{Hypothetical structures}

Let $G$ be a connected reductive group over the global field $K$, which need not be a function field.  Labesse \cite{Lab} has established a version of the following hypothesis over number fields; in what follows, we assume it is also available for function fields.  

\begin{hyp}\label{simpleBC}  (i)  Let $K'/K$ be a finite cyclic extension.  Let $\CG_0(G,K)$, resp. $\CG_0(G,K')$, denote the set of cuspidal automorphic representations of $G_K$, resp. $G_{K'}$.  There is a (non-empty) subset $\CG_{simple}(G,K) \subset \CG_0(G,K)$ with the following property.  For any $\Pi \in \CG_{simple}(G,K)$, there is a non-empty set $BC(\Pi) \subset \CG_0(G,K')$ with the following property:  for every place $v$ of $K$ at which both $\Pi$ and $K'$ are unramified, and for any place $v'$ of $K'$ dividing $v$, we have
$\Pi'_{v'} \isoarrow BC(\Pi_v)$ for any $\Pi' \in BC(\Pi)$, where $BC(\Pi_v)$ is the base change of $\Pi_v$ under the local Langlands correspondence for unramified representations.


(ii) Suppose $P \subset G$ is a rational parabolic subgroup with Levi factor $M$.  Suppose the global base change maps are defined for appropriate subsets $\CG_{simple}(M,K) \subset \CG_0(M,K)$ and $\CG_{simple}(G,K) \subset \CG_0(G,K)$.  Then  the  maps $\Pi \mapsto BC(\Pi)$ are compatible locally everywhere with parabolic induction from $M$ to $G$.
\end{hyp}

The subscript $_{simple}$ refers to the use of Arthur's simple  trace formula (STF) to establish global base change.  In \cite{Lab}, $K$ is a number field and the subset $\CG_{simple}(G,K)$ consists of representations that are Steinberg at some number of places (at least two).  This provides a sufficient condition to permit the application of the STF -- in other words to eliminate the difficult parabolic terms from both sides of the invariant trace formula.  The Steinberg condition also eliminates the endoscopic terms, thus substantially simplifying the comparison of trace formulas that implies Labesse's result.  We assume 

\begin{hyp}\label{labesse} Labesse's method works over function fields as well and we assume $\CG_{simple}(G,K)$ always includes at least the representations in Labesse's class.
\end{hyp}

It seems likely that this Hypothesis, combined with the general theory of automorphic representations of $GL(n)$, can provide a substitute for Hypothesis \ref{ACBC} in the discussion above for $GL(n)$, but we have not checked the details.  

Since we are not making any assumptions about packets for groups other than $GL(n)$, we introduce the following definition.

\begin{defn}.  Let $F'/F$ be a finite cyclic extension of non-archimedean local fields.  Let $\pi$ and $\pi'$ be irreducible admissible representations of $G(F)$ and $G(F')$, respectively.  We say that $\pi'$ is a base change of $\pi$ if there exists a finite cyclic extension $K'/K$ of global fields, a place $v$ of $K$ such that $K_v \isoarrow F$ and $K'_v \isoarrow F'$, and a cuspidal automorphic representation $\Pi \in \CG_{simple}(G,K)$ such that $\Pi_v \isoarrow \pi$ and, for some $\Pi' \in BC(\Pi)$, we have $\Pi'_v \isoarrow \pi'$.

More generally, if $F = F_0 \subset F_1 \subset \dots \subset F_r = F'$ is a sequence of finite cyclic extensions, with $\pi$ and $\pi'$ as above., we say that $\pi'$ is a base change of $\pi$ if there is a sequence $\pi_i$, with $\pi_0 = \pi$ and $\pi_r = \pi'$, such that $\pi_i$ is a base change of $\pi_{i-1}$ for each $i \geq 1$ in the above sense.
\end{defn}

The set of base changes of $\pi$ defined in this way thus may depend on the choice of global extension $K'/K$, as well as on the intermediate extensions in the setting of the second paragraph.  In particular, the set of base changes is potentially infinite.  

\begin{definition}\label{incorrigible}  Let $\pi \in \CG(G,F)$ be a supercuspidal representation of $G(F)$.  We say $\pi$ is {\bf incorrigible} if, for any sequence $F = F_0 \subset F_1 \subset \dots \subset F_r = F'$ of cyclic extensions, there is a supercuspidal representation $\pi' \in \CG(G,F')$ that is a base change of $\pi$.
\end{definition}

Say an admissible irreducible representation $\sigma$ of $G$ is {\it pure} if its Genestier-Lafforgue parameter $\phi_{\sigma}:  \Gamma_F \ra ^LG(C)$ has the property that, for any Frobenius element $Frob \in \Gamma_F$, the eigenvalues of $\phi_{\sigma}(Frob)$ are all Weil numbers of the same weight.  
Corresponding to Theorems \ref{gln} and \ref{Iwahori} we have the following Conjectures for general groups.  

\begin{conjecture}\label{noincorr}  There are no pure incorrigible supercuspidal representations.
\end{conjecture}

The purity hypothesis is meant to exclude cuspidal unipotent representations from consideration.  


The following refinement of Conjecture \ref{noincorr} should be a consequence of a version of the local Langlands conjecture for $G$ that includes  compatibility with parabolic induction.  

\begin{conjecture}\label{Iwahori2}  Let $\pi \in \CG(G,F)$ be a pure supercuspidal representation.  There is a finite sequence of cyclic extensions $F = F_0 \subset F_1 \subset \dots \subset F_r = F'$ such that every base change of $\pi$ to $F'$  contains an Iwahori-fixed vector.
\end{conjecture}

\section{Classical groups}\label{class}

With local base change defined loosely as in the previous section, we show that the doubling method of Garrett and Piatetski-Shapiro-Rallis, as refined by Lapid and Rallis in \cite{LR}, allows us to prove Conjecture \ref{noincorr} over local fields of positive characteristic.  We assume $G$ is a split classical group over the global function field $K$, of the form $Sp(2n)$ or $SO(V)$ for a vector space $V$ over $K$ with a non-degenerate symmetric bilinear form $b_V$.    We let $F = K_v$ for some place $v$. 

\begin{thm}\label{propclassical}  Assume Hypotheses \ref{simpleBC} and \ref{labesse} hold for $G$ over $K$.  Then $G$ has no incorrigible pure supercuspidal representations.
\end{thm}
\begin{proof}   We use the local theory of the doubling integral, as worked out in complete detail by Lapid and Rallis in \cite{LR} (see Remark \ref{doubb}).  If $\pi$ is an irreducible admissible representation of $G(E)$ for some local field $E$, $\psi:  E \ra \C^\times$ a continuous character, and $\omega:  E^\times \ra \C^\times$ a continuous (quasi)-character, we let $L(s,\pi,\omega)$ and $\varepsilon(s,\pi,\omega,\psi)$ denote the local $L$ and $\varepsilon$ factors associated to the data by the doubling method of \cite{psr}.  Similarly, if $L$ is a global field, $\Pi$ is a cuspidal automorphic representation of $G(L)$, and $\chi$ is a Hecke character of $L^\times\backslash \ad_L^\times$, then we let $L(s,\Pi,\chi) = \prod_w L(s,\Pi_w,\chi_w)$, where $w$ runs over places of $L$, and define the global $\varepsilon$ factor similarly.   We fix a global additive character 
$$\Psi = \otimes_w \psi_w:  K_{ad}/K \ra C^\times$$
For any finite separable extension $K'/K$ we let $\Psi_{K'} = \Psi \circ Tr_{K'/K}$ (but in general we drop the subscript $_{K'}$.

Now Theorem 4 of \cite{LR}, together with the results listed below, implies that the local and global $L$-functions provided by the doubling method satisfy the following  conditions.  

\begin{itemize}
\item[(i)]  Let $\rho:  {}^LG \ra GL(V_\rho)$ denote the standard representation of the Langlands $L$-group of $G$, and let $d(\rho) = \dim V_\rho$.  Suppose $\sigma$ is an unramified representation of $G(K_w)$ for some $w$ and $\alpha$ is an unramified character of $K_w^\times$.  Then $L(s,\sigma,\alpha) = L(s,\sigma,\alpha,\rho)$
is the unramified Langlands Euler factor attached to $\pi$, $\chi$, and $\rho$. 
\item[(ii)]  The global $L$-function $L(s,\pi,\chi) = \prod_w L(s,\pi_w,\chi_w)$ has a meromorphic continuation with at most finitely many poles, and satisfies a global functional equation
$$  L(1-s, \pi^\vee,\chi^{-1}) = \varepsilon(s,\pi,\chi)L(s,\pi,\chi)$$
\item[(iii)]  Fix $w$ and define the local gamma-factor
$$\gamma(s,\pi_w,\chi_w,\psi_w) = \varepsilon(s,\pi_w,\chi_w,\psi_w)\frac{L(1-s,\pi^{\vee}_w,\chi_w^{-1})}{L(s,\pi_w,\chi_w)}.
$$
Given two irreducible admissible representations $\pi_{1,w}$ and $\pi_{2,w}$, there is an integer $N$ such that, for all $\chi_w$ of conductor at least $N$, we have
$$\gamma(s,\pi_{1,w},\chi_w,\psi_w) = \gamma(s,\pi_{2,w},\chi_w,\psi_w).$$
\item[(iv)]  When $\pi$ is supercuspidal and $\omega$ is a unitary character, then the local factor $L(s,\pi, \omega)$ is holomorphic for $Re(s) \geq \frac{1}{2}$.  In particular, there is no cancellation between the poles of the numerator and denominator of the local gamma-factor $\gamma(s,\pi,\omega,\psi)$. 
\item[(v)]  There is an integer $d < d(\rho)$ such that, if $\pi$ is supercuspidal then the there are at most $d$ characters $\omega$ for which the local factor $L(s,\pi,\omega)$ has a pole, and each such pole is simple.  In particular, the local factor $L(s,\pi, \omega)$ has at most $d$ poles (counted with multiplicity).  
\end{itemize}

Point (iii) is the stability property proved by Rallis and Soudry (\cite{RS}; see also \cite{Br} for the analogous case of unitary groups).  Point (iv) is a consequence of Proposition 5 of \cite{LR}.  Point (v) is addressed at the end of the proof of Theorem 5.2 of \cite{Y}.  It is shown there that all the poles of the local Euler factor of a supercuspidal representation are poles of the local Euler factor there denoted $b(s,\chi)$.  An examination of the list on p. 667 of \cite{Y} confirms that $b(s,\chi)$ has fewer than $d(\rho)$ (fewer than $\frac{d(\rho) + 1}{2}$, in fact) and that they are all simple.  (See Remark \ref{polessc} below.)

Now suppose $\pi$ is pure supercuspidal.  By a simple reduction, using property (iv) of Corollary \ref{compatibilities}, we may assume $\pi$ has central character of finite order. 
Now suppose $\pi$ is also incorrigible.   Following \ref{glob}, we can globalize $\pi$ to a cuspidal automorphic representation $\Pi$ of $G_K$ with central character of finite order.   Let  $\phi_\Pi:  \Gamma_K \ra ^LG(C)$ denote its semisimple Langlands-Lafforgue parameter, $\phi_\pi:  \Gamma_F \ra ^LG(C)$ the local parameter, $\rho\circ\phi_{\pi}^{ss}$ the Frobenius-semisimplification of its composition with $\rho$.  Since $\pi$ is pure,  Grothendieck's monodromy theorem for representations of Weil groups of local fields implies that the image of inertia under $\rho\circ\phi_\pi^{ss}$ is of finite order, and necessarily solvable.  We can thus find a sequence $K = K_0 \subset K_1 \subset \dots K_i \subset K_{i+1} \subset K_r = K'$ of cyclic Galois extensions such that for every prime $v'$ of $K'$ dividing $v$, the image of the inertia group $I_{v'}$ under the restriction of $\rho\circ\phi_\pi^{ss}$ to $\Gamma_{K'}$ is trivial.  Since $\pi$ is incorrigible, and by Hypothesis  \ref{simpleBC} (i), it follows from our hypotheses that there is an automorphic representation $\pi'$ of $G_{K'}$ which is supercuspidal at some prime $v'$ of $K'$ dividing $v$, and whose Langlands-Lafforgue parameter satisfies
$$\phi_{\pi'} = \phi_\pi ~|~_{\Gamma_{K'}}.$$
Moreover, the semisimplification of the restriction of $\rho\circ\phi_{\pi'}$ to the decomposition group $\Gamma_{v'}$ is the sum of $d(\rho)$ unramified characters of the same Frobenius weight.  

Now $\rho\circ \phi_\Pi~|_{\Gamma_{K'}}$ is a continuous representation of the Galois group of the global function field $K'$.  Thus for any Hecke character $\chi$, $L(s,\rho\circ\phi_{\Pi},\chi)$ satisfies a functional equation of Galois type.  By the stability property (iii), the local $\gamma$ factors $\gamma(s,\pi_,\chi_w,\psi_w)$ at all places $w$ of $K'$ coincide with their Galois analogues.
In particular, there is a positive integer $d' \leq d$ such that there are $d'$ unramified characters $\omega_i$ of $(K'_{v'})^\times$, all of the same weight, such that, for each $i$,  the local Euler factor 
$$L(s,\pi',\omega_i) = L(s,\phi_{\pi'}~|~_{\Gamma_{K'_{v'}}}, \omega_i)$$ 
has a pole of multiplicity $d'_i$ at $s = 1$, and such that $\sum_i d'_i = d(\rho)$.  But this contradicts point (v).  Thus $\pi'$ cannot be supercuspidal.
\end{proof}

\begin{remark}\label{doubb}  There are no references for the doubling method over function fields.  The proof of Theorem \ref{propclassical} assumes that it works in exactly the same way for function fields as in \cite{psr,LR}.  The Langlands-Shahidi method has been developed by Lomel\'i over function fields.  Section 7 of \cite{Lo18} contains details about the Langlands-Shahidi local factors for classical groups; note that special care has to be taken in characteristic $2$.  Since we have no information a priori about generic representations (the results of \cite{GV} on the tempered packet conjecture rely on Arthur's results in \cite{A}, which we have deliberately chosen not to use), the Langlands-Shahidi method is not available to us.
\end{remark}

\begin{remark}\label{polessc}  The possible poles of local Euler factors of supercuspidal representations of unitary groups are determined in Theorem 6.2 of \cite{HKS}.  They correspond as expected to the classification of cuspidal unipotent representations.   In particular, it follows from the characterization in \cite{HKS} that the local Euler factor of a pure supercuspidal representation has at most a single simple pole.  (It follows from the actual classification that a supercuspidal representation whose standard local $L$-factor has even one pole cannot be pure, but it would be circular to admit the actual classification at this stage of the argument.)    The  case of orthogonal and symplectic groups is not made explicit in \cite{Y} but the argument of \cite{HKS}, applied to his computations, suffices to verify the claim.  \end{remark}

Of course, if we are assuming Hypothesis \ref{simpleBC} then we could also assume the full stable twisted trace formula, as in Arthur's book \cite{A}, and then Proposition \ref{propclassical} follows from the case of $GL(n)$ by Arthur's trace formula arguments.   However, base change, in the form of Hypothesis \ref{labesse}, is considerably simpler to manage.

\subsection*{The local correspondence}  Let $G$ be a split classical group and let $\rho:  W_F \ra \hat{G}$ be a local Langlands parameter for $G$.  We say $\rho$ is {\it irreducible} if the image of $\rho$ is contained in no proper parabolic subgroup of $\hat{G}$, and we say $\rho$ is {\it stable} if its centralizer in $\hat{G}$ is trivial.

Starting with Theorem \ref{propclassical} we can try to follow the steps of Section \ref{localGLn} to give a complete proof of the local correspondence between $L$-packets of supercuspidal representations of a split classical group $G$ and pure irreducible Langlands parameters with values in $\hat{G}$ (we side-step the problem of unipotent supercuspidal parameters, which require other methods).  The argument is necessarily more complicated, because of the presence of $L$-packets.  If we admit endoscopic transfer (thus we require the full strength of the stable trace formula, twisted as well as untwisted), an induction argument reduces the proof to the following situation:  $F'/F$ cyclic of prime degree, $\rho:  W_F \ra \hat{G}$ a stable irreducible parameter, such that $\rho' = \rho~|_{W_{F'}}$ is reducible and each of the irreducible components of $\rho'$ is stable.  Then $\rho'$ is the parameter of a (unique) tempered representation of $G(F')$ that is invariant under $Gal(F'/F)$, and we need to relate its twisted character to the character of a supercuspidal representation of $G(F)$.  This is the generalization of (b) of Theorem \ref{ACBCloc}.

\section{Automorphic descent for classical groups and $G_2$}  

Continuing the exercise of seeing how much information can be extracted from a given collection of hypotheses, we now show how to derive the surjectivity of the local parametrization of Corollary \ref{compatibilities} from a version of the automorphic descent method of Ginzburg-Rallis-Soudry.    We state a special version of the main theorem of \cite{GRS}.   Let $K$ be a number field and let $m \geq 1$ be a positive integer; let $GL(m)$ denote the group $GL(m)$ over $K$.  Let
$V = V_m/K$ be a finite dimensional vector space.  If $m = 2n+1$ we assume $\dim V_m = 2n$ and $V_v$ is endowed with  a non-degenerate symplectic form, and we let $\alpha_V$ denote the representation $Sym^2$ of $\hat{G}_m = GL(m)$.  If $m = 2n$ we distinguish two subcases:  either $\dim V_m = 2n+1$ and $\alpha_V$ is the representation $\wedge^2$ of $\hat{G}_m$, or $\dim V_m = 2n$ and $\alpha_V$ is the representation $Sym^2$.  Let $G = G_V$ denote the identify component of the symmetry group of this form.  For simplicity we assume $G$ to be a split group; however, the results of \cite{GRS} are valid for quasi-split groups, including unitary groups.  

Thus we restrict our attention to cases (1), (2), and (10) in the lists (3.42) and (3.43) of \cite{GRS}.   In each of these cases there is a canonical $L$-homomorphism
\begin{equation}\label{Lh}
j_V:  {}^LG \ra {}^LGL(m).
\end{equation}
It thus makes sense to say that an automorphic representation $\tau$ of $GL(m)$ is a weak functorial lift of an automorphic representation $\pi$ of $G$:  this means that, at almost every finite place $v$ of $K$, the local components $\tau_v$ and $\pi_v$ are unramified, and the Satake parameter of $\tau_v$ is obtained from that of $\pi_v$ by composition with the morphism $j_V$.

\begin{thm}[Theorem 3.1, \cite{GRS}]\label{descent}   With notation as above, let $\tau$ be a unitary cuspidal automorphic representation of $GL(m)$ over $K$.  Suppose the Langlands $L$-function $L(s,\tau,\alpha)$ has a pole at $s = 1$.  Assume the central character $\omega_\tau$ of $\tau$ is trivial, except when $G = SO(2n)$ (in which case it is automatic that $\omega_\tau^2$ is trivial).  Then there is a multiplicity-free cuspidal automorphic representation $\pi$ of $G_V$ such that $\tau$ is a weak functorial lift of $\pi$ with respect to \eqref{Lh}.

Moreover, each irreducible summand of $\pi$ is globally generic with respect to some non-degenerate character of the unipotent radical of a Borel subgroup of $G_V$.  
\end{thm}

Joseph Hundley and Baiying Liu have recently announced an analogous result for the split exceptional group $G_2$.  Let $j_2:  G_2 \ra GL(7)$ denote the irreducible $7$-dimensional representation of $G_2$. 

\begin{thm}[Hundley-Liu, to appear]\label{HLiu}  Let $\tau$ be a unitary cuspidal automorphic representation of $GL(7)_K$.  Suppose the Langlands $L$-function $L(s,\tau,\wedge^3)$ has a pole at $s = 1$.  Then (under mild hypotheses on the unramified components of $\tau$) there is a globally generic cuspidal automorphic representation $\pi$ of $G_2(\ad_K)$ such that $\tau$ is a weak functorial lift of each irreducible constituent of $\pi$ with respect to the $L$-homomorphism $j_2$.  
\end{thm}

\begin{hyp}\label{descenthyp}  Let $K$ be a global function field over the finite field $\Fp$.   Let $G$ be one of the split groups $SO(2n)$, $SO(2n+1)$, $Sp(2n)$ or $G_2$, viewed as an algebraic group over $K$.  Assume the validity for $G$ of either Theorem \ref{descent} or Theorem \ref{HLiu} (whichever of the two is relevant), for any finite extension of $K$.
\end{hyp}

The following is well-known and is proved by standard approximation arguments.

\begin{lemma}\label{globgal}  Let $F$ be a local field and let $\rho:  \G_F \ra GL(n,\Qlb)$ be a homomorphism with finite image.  Then there is a global field $K$ with place $v$ such that $K_v \isoarrow F$, an inclusion $\G_F \isoarrow \Gamma_v \subset \G_K$ as decomposition group, and a homomorphism $\sigma:  \G_K \ra GL(n,\Qlb)$ such that $\sigma~|_{\Gamma_v} \simeq \rho$, and such that the restriction of $\sigma$ to $\Gamma_v$ identifies the images of $\sigma$ and $\rho$.
\end{lemma}

When $G = G_2$ we let $m = 7$, $\alpha = \wedge^3$, and let $j_G$ be the map $j_2$ introduced above; when $G$ is classical we let $j_G = j_V$.
\begin{prop}\label{surj}   Assume  Hypothesis \ref{descenthyp} for the group $G$.  Let $k$ be a finite field of order $q = p^f$ for some $f$, let $F = k((T))$, and let 
$$\rho_G:  \G_F \ra \hat{G}(\Qlb)$$
be a Langlands parameter for $G$.  Suppose 
$$\rho = j_G \circ \rho_G:  \G_F \ra GL(m,\Qlb)$$
is an irreducible representation of $\G_F$.
 Then there is a generic supercuspidal representation $\pi$ of $G(F)$ such that 
$$\CL_F(\pi) = \rho_G.$$
\end{prop} 
\begin{proof}  Let Let $K$, $v$, and $\sigma:  \G_F \ra GL(m,\Qlb)$ be as in Lemma \ref{globgal}.   By the global Langlands correspondence of \cite{Laf02} for $GL(m)$ there is an automorphic representation $\tau$ of $GL(m)$ such that $\CL^{ss}_K(\tau)= \sigma$.  Moreover, since $\rho$ is irreducible, so is $\sigma$, so $\tau$ is cuspidal.  On the other hand, the image $\sigma(\G_K) \subset j_G(\hat{G})$.  It follows that $\G_F$ fixes a line in the representation $\alpha \circ \sigma$, hence that
$$L(s,\tau,\alpha) = L(s,\alpha\circ\sigma)$$
has a pole at $s = 1$.  We may thus apply Hypothesis \ref{descenthyp} to obtain a globally generic cuspidal automorphic representation $\pi$ of $G(\ad_K)$ such that $\tau$ is a weak functorial lift of  each irreducible constituent of $\pi$ with respect to the $L$-homomorphism $j_G$.  By Chebotarev density, it then follows that
$$\CL^{ss}_F(\pi_v) = \rho_G.$$
Since $\rho_G$ is irreducible, it follows from (ii) of Theorem \ref{compatibilities} that $\pi_v$ is supercuspidal.
\end{proof}

\begin{remark}  When $G$ is a classical group, the hypothesis that $\rho$ be irreducible is superfluous: the main result of \cite{GRS} provides descent for automorphic representations of $GL(m)$ that are not necessarily cuspidal, under  conditions that are implied by the irreducibility of $\rho_G$.  Presumably when $G = G_2$, $\rho$ is irreducible unless $\rho_G$ is an endoscopic parameter, in which case the proposition can be obtained by endoscopic lift -- or, more concretely, by the exceptional theta correspondence for the pair $(PGL(3),G_2)$ \cite{GS}.   
\end{remark}

\subsection*{Surjectivity vs. incorrigibility}
The possibility of using the results of \cite{GRS} to prove Proposition \ref{surj} was noted in \S 7.4 of \cite{GLo};  they write that ``The theory of local descent should continue to work over a local function field $F$. However, it is presently not written up in this generality in the literature."  The results of Hundley and Liu are also only proved over number fields.  Liu has indicated that their proof should work for $G_2$ over fields of sufficiently large characteristic, since they use properties of unipotent conjugacy classes in $E_7$.  

Even assuming Hypothesis \ref{descenthyp}, however, it is not clear that we know enough to prove Conjecture \ref{noincorr} for $G_2$.   For this we would need a full local theory of the integral representations  \cite{GH,GS1,GS2} of the standard $L$-function of $G_2$, satisfying conditions (i)-(v) of the proof of Theorem \ref{propclassical}.

\section{Formal degree and incorrigible representations}

The article \cite{HII08} of Hiraga, Ichino, and Ikeda proposed an explicit conjectural formula for the formal degrees of discrete series representations of reductive groups over local fields and a related formula for the Plancherel measure.   The formula has been proved in a great many cases but remains open, with few results known for exceptional groups -- which is hardly surprising, since the conjecture is formulated in terms of the local Langlands parametrization.  In this final section we show that Conjecture \ref{noincorr} for groups over fields of positive characteristic is a simple consequence of the Hiraga-Ichino-Ikeda conjecture.   For the local Langlands parametrization we take the one defined in \cite{GLa}.  

The careful reader may object that \cite{HII08} only states a conjecture for groups over local fields of characteristic zero, but Ichino has assured us that the conjecture can be stated just as well for groups over fields of positive characteristic.  The assumption of characteristic zero was only made in \cite{HII08} because the proofs given there in specific examples were based on methods that at the time were only available for $p$-adic fields.  (Presumably the article \cite{Lo18} allows for an extension of the proofs in \cite{HII08} to positive characteristic.)

We recall a simplified version of the conjecture of Hiraga-Ichino-Ikeda.

\begin{conjecture}[\cite{HII08}, Conjecture 1.4]\label{HII}  Let $\phi: Gal(\bar{F}/F) \ra {}^LG$ be an  elliptic tempered Langlands parameter.  Then for any $\pi$ in the $L$-packet of $\phi$, the formal degree $d(\pi)$ is a non-zero constant multiple of the $\gamma$-factor $|\gamma(0,Ad\circ\phi,\psi)|$, where
$$|\gamma(s,Ad\circ\phi,\psi)| = |\varepsilon(s,Ad\circ\phi,\psi) \frac{L(1-s,Ad\circ \hat{\phi})}{L(s,Ad\circ\phi)}|.$$
\end{conjecture}
In the statement, $Ad$ is the adjoint representation of ${}^LG$ on its Lie algebra and $\hat{\phi}$ is the parameter of the contragredient $\hat{\pi}$.  The formula in \cite{HII08} is completely explicit; the non-zero constant reflects the place of $\pi$ in its $L$-packet.  

\begin{prop}\label{HIIII}  Conjecture \ref{HII} for supercuspidal representations implies Conjecture \ref{noincorr}.
\end{prop}

\begin{proof}  Suppose $\phi$ is pure unramified and tempered.  Then the centralizer $Z_{\hat{G}}(\phi)$ of $\phi$ in $\hat{G}$ is of positive dimension, so $Ad\circ\phi$ acts trivially on a positive-dimensional subspace of the Lie algebra of $\hat{G}$.  On the other hand, since $Ad \circ \phi$ is pure of weight $0$, the factor $L(1-s,Ad\circ \hat{\phi})$ has no pole at $s = 0$.  It then follows that $L(s,Ad\circ\phi)$ has a pole at $s = 0$, so $|\gamma(0,Ad\circ\phi,\psi)| = 0$.  Thus no multiple of $|\gamma(0,Ad\circ\phi,\psi)|$ can be the formal degree of a discrete series representation.

Now suppose $\pi$ is a pure  supercuspidal representation of $G(F)$, with Genestier-Lafforgue parameter $\phi$.   Let $F'/F$ be a (solvable) Galois extension such that $\phi' = \phi |_{\Gamma_{F'}}$ is unramified. Any base change $\pi'$ of $\pi$ to $G(F')$ then has Genestier-Lafforgue parameter $\phi'$, by the Chebotarev density argument we have already used.  Thus by the argument of the previous paragraph, the Hiraga-Ichino-Ikeda Conjecture \ref{HII} implies that $\pi'$ cannot be a discrete series representation.
\end{proof}

\begin{question}  Is the converse to Proposition \ref{HIIII} true?
\end{question}

\end{document}